\documentclass[twoside,11pt]{article}
\RequirePackage[OT1]{fontenc}
\RequirePackage{amsthm,amsmath}
\RequirePackage{graphicx}
\RequirePackage[numbers]{natbib}
\RequirePackage[colorlinks,citecolor=blue,urlcolor=blue]{hyperref}
\usepackage[T1]{fontenc}
\usepackage{amssymb}
\usepackage{pdfsync}
\usepackage{dsfont}
\numberwithin{equation}{section}
\newtheorem{thm}{Theorem}[section]

\theoremstyle{definition}
\newtheorem{dfn}{Definition}[section]

\theoremstyle{remark}
\newtheorem{rmk}{Remark}[section]
\allowdisplaybreaks
\def\R{\mathbb{R}} % r\'eels
\def\1{~\mbox{I\hspace{-.6em}1}} % fonction carct\'eristique
\def\E{\mathds{E}} %

\def\P{\mathds{P}} %
\def\v{\mbox{Var\,}}

\def\cov{\mbox{Cov}}%
\usepackage{vmargin}            % redéfinir les marges 
\setmarginsrb{3cm}{3cm}{3cm}{3cm}{1cm}{1cm}{1cm}{1cm} 

\numberwithin{equation}{section}
\theoremstyle{plain}
\begin{document}

%\begin{frontmatter}
\title{Exponential inequalities for unbounded functions of geometrically ergodic Markov chains. Applications to quantitative error bounds for regenerative Metropolis algorithms.}
%\thankstext{T1}{Footnote to the title with the ``thankstext'' command.}
\date{}
%\begin{aug}
\author{ Olivier Wintenberger\footnote{{\tt olivier.wintenberger@upmc.fr},
 Sorbonne Universit\'es, UPMC Univ Paris 06,
       LSTA, Case 158,
4 place Jussieu,
75005 Paris,
FRANCE 
\& Department of Mathematical Sciences,
University of Copenhagen, DENMARK}}

%\printead{e1}\\
%\printead{u1}}
%\end{aug}
\maketitle

\begin{abstract}
The aim of this note is to investigate the concentration properties of unbounded functions of geometrically ergodic Markov chains. We derive concentration properties of centered functions with respect to the square of Lyapunov's function in the drift condition satisfied by the Markov chain. We apply the new exponential inequalities to derive confidence intervals for MCMC algorithms. Quantitative error bounds are provided for the regenerative Metropolis algorithm of \cite{brockwell:kadane:2005}.
\end{abstract}

%\begin{keyword}[class=AMS]
%\kwd[Primary ]{60E15}
%\kwd[; secondary ]{28A35}
%{\bf AMS MSC 2010:} 60E15; 28A35
%{62J05} {62M10} {62M20}
%\end{keyword}

%\begin{keyword}
{\small {\bf Keywords:} Markov chains, exponential inequalities, Metropolis algorithm, Confidence interval.}
%\end{keyword}

%\end{frontmatter}

\section{Introduction}\label{sec1}

At the conference in honor of Paul Doukhan, J\'er\^ome Dedecker presented the new Hoeffding inequality  \cite{dedecker:gouezel:2015} for functions $f$ of a geometric ergodic Markov chain $(X_k)$, $1\le k\le n$. Using a similar counter example as in Section 3.3 of \cite{adamczak:2008}, he  showed that the boundedness condition on the function $f$ is necessary to obtain such exponential inequalities for functions of geometrically ergodic Markov chain.\\

In this note, we introduce new  deviation inequalities for relaxing the boundedness condition. We extend the framework of \cite{dedecker:gouezel:2015} by considering concentration properties of $f$ involving a second order term that depends on the Markov chain. Such exponential inequalities are called {\em empirical} Bernstein inequalities and used to derive observable confidence intervals in the machine learning literature \cite{audibert:2007}. As the second order term is an over-estimator of the asymptotic variance, the new inequality \eqref{mres} is also closely related to the self-normalized concentration inequalities studied in \cite{delapena:klass:lai:2004}. The novelty of our result is the appearance of the second-order term depending on the squares of the Lyapunov function $V$ in the drift condition \eqref{eq:geom} satisfied by the Markov chain.\\

The new deviation inequality \eqref{mres} is remarkably simple as there is only one correction term, an empirical  over-estimator of the variance. Even for bounded functions, existing Bernstein-type inequalities for ergodic Markov chains contain an extra logarithmic correction term compared with the iid case; see \cite{bertail:clemencon:2010} and \cite{adamczak:bednorz:2015}. This correction term, coming from the concentration properties of the regeneration length, is necessary \cite{adamczak:2008}. For unbounded functions, another correction term  appears in Fuk-Nagaev-type inequalities due to the marginal tails. This correction term is also necessary in the iid case for additive functionals; see \cite{nagaev:1963}. Thus, the observable second order term in the empirical Bernstein inequality \eqref{mres} encompasses three necessary corrections due to different causes: the asymptotic variance, the random regeneration length and the heavy-tailed marginal distribution. The price to pay are the constants in \eqref{mres} that become enormous even in toy examples.\\

We apply our result to the construction of confidence intervals for some MCMC algorithms. Previous studies are based on a two-step reasoning:  first, some bounds are derived with unknown constants, via Chebyshev or Hoeffding inequalities; see \cite{latuszynski:miasojedow:niemiro:2013} and \cite{gyori:paulin:2012} respectively. The second step consists of over-estimating the constants. Our new empirical exponential inequality provides concentration properties thanks to a second-order term that is observed. We achieve a quantitative error analysis by a direct application of the techniques in \cite{delapena:klass:lai:2004} for the Regenerative Metropolis Algorithm of \cite{brockwell:kadane:2005}. A similar one-step procedure was developed in \cite{joulin:ollivier:2010} under a more restrictive Ricci curvature condition. Our approach provides quantitative bounds that can be reasonable if Lyapunov's function can be well-chosen. However, the confidence intervals are certainly over-estimated due to the coupling approach used in the proof inducing enormous constants.\\

The paper is organized as follows. The main result, the concentration properties for unbounded functions of Markov chains, is stated  in Theorem \ref{th:main} of Section \ref{sec:main}. Its proof is given next. It relies on a coupling approach combining the arguments of \cite{hairer:mattingly:2011} and \cite{wintenberger:2015}. Then Section \ref{sec:appl} is devoted to the construction of confidence intervals for MCMC algorithms. The case of the regenerative Metropolis algorithm of \cite{brockwell:kadane:2005} is studied in detail. Simulation study and discussion on the applications are given in Section \ref{sec:disc}.

\section{Concentration properties  for unbounded functions of Markov chains under the drift condition.}\label{sec:main}

We consider an exponentially ergodic Markov kernel $P$ on   some countably generated space $E$  that satisfies the following drift and minorization conditions  \eqref{eq:geom} and \eqref{eq:ir} respectively: there exist a Lyapunov function $V:E\mapsto [1,\infty)$, a  probability measure $\nu $, positive constants $b$, $R_0$ and $\beta<1$, a function $c:[R_0,\infty)\to(0,\infty)$ such that
\begin{align}
\label{eq:geom} PV\le \beta V+b,& \\
\label{eq:ir}P(x,\cdot) \ge c(R) \nu(\cdot),&\qquad \mbox{if}\qquad  V(x)\le R,\qquad R\ge R_0.
\end{align}
These conditions are slightly stronger than the exponential  ergodicity of the Markov chain. It is related to the Feller property; see \cite{meyn:tweedie:1993}. In particular, it requires strong aperiodicity. This, however, is not a problem in applications: the conditions \eqref{eq:geom} and \eqref{eq:ir} are satisfied in many examples, such as random coefficient autoregressive processes; see \cite{feigin:tweedie:1985}, or the trajectories of the Random Walk Metropolis algorithm, see Section \ref{sec:appl}.
Let us consider a function $f$ on $E^n$ satisfying for some $L_k>0$,
\begin{equation}\label{eq:cond}
|f(x_1,\ldots,x_n)-f(x_1,\ldots,x_{k-1},y_k,x_{k+1},\ldots,x_n)|\le L_k(V (x_k)+ V(y_k)).
\end{equation}
Dedecker and Gouezel \cite{dedecker:gouezel:2015} extended the Hoeffding inequality to the   trajectory $(X_1,\ldots,X_n)$ of the Markov chain with transition probability $P$ starting from $X_0=x$ and distributed as $\P_x$ in the case $V=2^{-1}$. They proved the existence of a constant   $K_R>0$ independent on $n$ such that
$$
\E_x[\exp(f-\E_x[f])]\le e^{K_R\sum_{k=1}^nL_k^2},\qquad x\in \{V\le R\}.
$$
We prove the following result:
\begin{thm}\label{th:main}
Assume that $P$ satisfies the drift conditions \eqref{eq:geom} and the minorization condition \eqref{eq:ir} with $R\ge R_0$ satisfying $\bar\beta(R):=\beta+2b/(1+R)<1$. Assume that $PV_k^{2}:=\E[V^{2}(X_k)\mid X_{k-1}]$ is well defined and denote  $V_k:=V(X_k)$. Then there exist coefficients $\gamma_{k,0}(1)\ge 0$, $k\ge 0$, satisfying
\begin{equation}\label{eq:K}
\sum_{k=1}^\infty\gamma_{k,0}(1)\le K:= \frac{1+\bar \beta(R) ((R -1)/c(R)-R )}{1-\bar\beta(R) },
\end{equation}
and for any $f$ satisfying \eqref{eq:cond}, $x\in E$ and $\lambda\in\R$:
\begin{equation}\label{mres}
\E_x\Big[\exp\Big(\lambda(f-\E_x[f])-\frac{\lambda^2}2\sum_{k=1}^n\Big(\sum_{j=k}^n\gamma_{k-j,0}(1)L_j\Big)^2(PV_k^{2}+V_k^{2})\Big)\Big]\le 1.
\end{equation}
Eq. \eqref{mres} holds true in the stationary case with $\E$ replacing $\E_x$ and $\E[V(X_1)^{2}]$ replacing $PV_1^{2}$.
\end{thm}
\begin{rmk}
In the iid case, from Definition \ref{def:wd} below, the coefficients $\gamma_{k,0}(1)$ are null for $k>0$ and $\gamma_{0,0}(1)=1$. Thus, we obtain an extension of the McDiarmid inequality \cite{mcdiarmid:1989} that seems to be new:
\[
\E\Big[\exp\Big(\lambda(f-\E[f])-\frac{\lambda^2}2\sum_{k=1}^nL_k^2(\E[V(X_k)^{2}]+V(X_k)^{2})\Big)\Big]\le 1,\qquad \lambda\in \R.
\]
Notice that under the bounded differences condition $V=2^{-1}$  we recover the optimal constant $8^{-1}$.
\end{rmk}
\begin{rmk}
The inequality \eqref{mres} implies exponential inequalities for the normalized process  when $f=\sum_{k=1}^nV_k$. Applying Theorem 2.1 of \cite{delapena:klass:lai:2004}, we obtain the subgaussian inequality $\E[\exp(xY)]\le \sqrt 2 \exp(C x^2)$, $x>0$, of the process 
$$
Y:=\frac{f-\E_x[f]}{\sqrt{\sum_{k=1}^n(PV_k^{2}+V_k^{2}+2\E_x[V_k^2])}}
$$
for some constant  $C>0$. Such bounds cannot be obtained using the approach of \cite{dedecker:gouezel:2015} because the bounded differences properties \cite{mcdiarmid:1989} of such self-normalized processes are growing as $\sqrt n$.
\end{rmk}
\begin{rmk}For a bounded function $f$ one can compare \eqref{mres} with the result of Dedecker and Gouezel \cite{dedecker:gouezel:2015}. The limitation of the result in Theorem \ref{th:main} is that considering $V=2^{-1}$ constrains the Markov chain to be uniformly ergodic. In such a restrictive case, the classical Bernstein inequality was extended by Samson in \cite{samson:2000} and the Hoeffding inequality by Rio in \cite{rio:2000}.
\end{rmk}
\begin{proof}[Proof of Theorem \ref{th:main}] The proof is based on a new coupling argument applied to the coupling scheme $(X_k,X_k')_{1\le k\le n}$ of \cite{rosenthal:2003}, where $(X_k')_{1\le k\le n}$ is a copy of $(X_k)_{1\le k\le n}$.  For completeness, let us first recall the construction of the coupling scheme.  Any Markov chain $\bar P$ on $E^2$ with common marginal $P$ also satisfies
$$
\bar P \bar V(x,x')\le \beta \bar V(x,x') +2b,
$$
for the drift function $\bar V(x,x')=V(x)+V(x')$.  Moreover, there exists a coupling kernel $\bar P $, see \cite{rosenthal:2003} for details, with common marginal $P$ such that 
$$
\bar P ((x,x'),\cdot\times \cdot)\ge c(R) \nu(\cdot),\qquad (x,x')\in \{  V\le  R\}^2.
$$
In particular, $\bar P ((x,x'),\cdot)$, $(x,x')\in \{ V\le R\}^2$, has a mass at least equal to $c(R)$ on the diagonal. As $\bar V\ge 1+R$ when $(x,x')\notin \{ V\le R\}^2$, we also have
$$
\bar P \bar V(x,x')\le \left(\beta +\frac{2b}{1+R}\right)\bar V(x,x'),\qquad (x,x')\notin \{ V\le R\}^2.
$$
We have $\bar \beta =\beta + 2b/(1+R)<1$ by assumption.
Then one can apply the Nummelin splitting scheme on the Markov chain $(X_t,X_t')$ driven by $\bar P $. There exists an enlargement $(X_t,X_t',B_t)$ with $B_t\in \{0,1\}$ such that it admits an atom $A=\{ V\le R\}^2\times \{1\}$ and $\P(B_t=1\mid (X_t,X_t')\in \{ V\le R\}^2)=c(R) $. Let $\tau$ and $\tau_A$ denote the first hitting time to $\{ V\le R\}^2$ and the atom $A$, respectively. Due to the regenerative properties of the enlarged chain, one can always restart $(X_t,X_t',B_t)$ such that $X_t=X_t'$ for $t\ge \tau_A(\ge \tau)$. From the Dynkin formula (Theorem 11.3.1 of \cite{meyn:tweedie:1993}), denoting $\bar V_k=\bar V(X_k,X_k')$ and $\bar P \bar V_k=\E[\bar V(X_k,X_k')\mid (X_{k-1},X_{k-1}')]$ we have
$$
\bar \E_{x,x'}[\bar V_{\tau}]=\bar V(x,x')+\bar \E_{x,x'}\Big[\sum_{k=1}^{\tau} \bar P \bar V_k-\bar V_{k-1} \Big].
$$
Plugging the drift condition in this formula, we obtain
\[
\bar \E_{x,x'}[\bar V_{\tau }]\le \bar V(x,x')+( \bar \beta(R)-1)\bar \E_{x,x'}\Big[\sum_{k=1}^{\tau } \bar V_{k-1} \Big],
\]
that yields
\begin{equation}\label{eq:dynkin}
\E_{x,x'}\Big[\sum_{k=0}^{\tau} \bar V_{k} \Big]\le \frac{\bar V(x,x') -\bar \beta(R)\bar \E_{x,x'}[\bar V_{\tau }]}{1- \bar \beta(R)}\le \frac{\bar V(x,x')-2\bar \beta(R)}{1- \bar\beta(R)}.
\end{equation}
Denoting $\tau(j)$ the successive hitting times to $\{ V\le R\}^2$,  we have
\begin{align*}
\bar \E_{x,x'}\Big[\sum_{k=0}^{\tau_A} \bar V_{k}\Big]&=\E_{x,x'}\Big[\sum_{k=0}^{\tau} \bar V_{k} \Big]+ \bar \E_{x,x'}\left[\sum_{j=1}^{\infty} \sum_{k=\tau(j)+1}^{\tau(j+1)}\bar V_k1_{B_1=\cdots B_j=0} \right]\\
&\le \frac{\bar V(x,x')-2\bar \beta(R)}{1- \bar\beta(R)}+  \bar \E_{x,x'}\left[\sum_{j=1}^{\infty} (1-c(R))^j\bar E_{(X_{\tau(j)},X_{\tau(j)}')}\left[\sum_{k=\tau(j)+1}^{\tau(j+1)}\bar V_k  \right]\right].
\end{align*}
using the strong Markov property to assert the last inequality. Using \eqref{eq:dynkin} and $\sup_{\{V\le R\}^2}\bar V\le 2R$ we obtain,
\[
\bar \E_{(X_{\tau(j)},X_{\tau(j)'})}\left[\sum_{k=\tau(j)+1}^{\tau(j+1)}\bar V_k  \right]\le \sup_{\{V\le R\}^2}\E_{x,x'}\Big[\sum_{k=1}^{\tau} \bar V_{k} \Big]\le\frac{2\bar\beta(R) }{1-\bar\beta(R)}(R-1).
\]
Collecting those bounds, we derive
\begin{align*}
\bar \E_{x,x'}\Big[\sum_{k=0}^{\tau_A} \bar V_{k}\Big]&\le \frac{\bar V(x,x')-2\bar \beta(R)}{1-\bar\beta(R)}+\frac{2\bar\beta(R) (R-1)}{c(R)(1-\bar\beta(R))}-\frac{2\bar\beta(R)(R-1)}{1-\bar\beta(R)}\\
&\le \frac{\bar V(x,x')-2\bar \beta(R) R}{1-\bar\beta(R)}+\frac{2\bar\beta(R) (R-1)}{c(R)(1-\bar\beta(R))}.
\end{align*}
We are now ready to use our new coupling argument, combining the metric  $d_V(x,y)=\bar V(x,y)=(V(x)+V(y))\1_{x\neq y}$  of \cite{hairer:mattingly:2011}  with the $\Gamma$-weak dependence notion of \cite{wintenberger:2015}. A main difference with \cite{hairer:mattingly:2011} is that the coupling argument of \cite{wintenberger:2015} does not require any contractivity of the Markov kernel with respect to the metric $d_V$. We obtain
\begin{equation}\label{eq:wd}
\bar\E_{x,x'}\Big[\sum_{k=0}^{\infty}d_V(X_k,X_k')\Big]=\bar\E_{x,x'}\Big[\sum_{k=0}^{\tau_A}d_V(X_k,X_k')\Big]\le K d_V(x,x')
\end{equation}
with $K$ defined in \eqref{eq:K} as $X_k=X_k'$ for $k>\tau_A$. Recall the following definition from \cite{wintenberger:2015}:
\begin{dfn}\label{def:wd}
A Markov chain is $\Gamma_{d_V,d_V}(1)$-weakly dependent  if  there exist  coefficients $\gamma_{k,0}(1)\ge 0$ such that for any $(x_0,x'_0)\in E^{2}$ there exists a coupling scheme $(X_k,X_k')_{1\le k\le n}$ conditionally on $(X_0,X_0')=(x_0,x'_0)$ satisfying
$$
\E_{x_0,x'_0}[ d_V(X_k,X_k')]\le  \gamma_{k,0}(1)\, d_V(x_0,x'_0),\qquad 0\le k\le n.
$$
\end{dfn}
In view of \eqref{eq:wd}, we claim that the Markov chain $(X_k)_{1\le k\le n}$ is $\Gamma_{d_V,d_V}(1)$-weakly dependent with dependence coefficients satisfying
$$
\sum_{k=0}^\infty \gamma_{k,0}(1)\le K.
$$
By $X$ we denote the trajectory $(X_1,\ldots,X_n)$ on $E^n$ starting from $x$ with distribution $\P_x$ and by $d_{V,L}$ the metric on $E^n$ such that 
$$d_{V,L}(x,y)=\sum_{k=1}^n L_kd_V(x_k,y_k),\qquad x,y\in E^n.$$
Recall the definition of the Wasserstein distance between $\P_x$ and any measure $Q$ on $E^n$
$$
W_{1,d_{V,L}}(\P_x,Q)=\inf_\pi \E_\pi[d_{V,L}(X,Y)],
$$
where $\pi$ is any coupling measure such that $(X,Y)\sim \pi$, $X\sim \P_x$ and $Y\sim Q$. We require more notation from \cite{wintenberger:2015}; For any deterministic $\alpha=(\alpha_1,\ldots,\alpha_n)\in (\R^+)^n$ we denote
\[
\tilde W_{\alpha,d_{V}}(\P_x,Q)=\inf_\pi \E_\pi\Big[\sum_{k=1}^n \alpha_kd_{V}(X_k,Y_k)\Big].
\]
For any $Y\in E^n$ distributed as $Q$, $Q_{Y^{(j-1)}}$ denotes the conditional distribution of $Y_j$ given $Y^{(j-1)}=(Y_1,\ldots,Y_j)$ for $1\le j\le n$ (artificially considering that $y_0=x$, the initial state of $\P_x$).
Equality l.7 p.15 of \cite{wintenberger:2015} in the specific case $p=1$ and $d=d'=d_V$, states that for any $\alpha$, we have
\begin{eqnarray*}
\tilde W_{\alpha,d_{V}}(\P_x,Q)
&\le &  \sum_{j=1}^n\sum_{k=j}^n\alpha_k\gamma_{k-j,0}(1)
\E_Q[W_{1,d_V}(P_{Y^{(j-1)}},Q_{Y^{(j-1)}})]\\
&\le &  \sum_{j=1}^n\sum_{k=j}^n\alpha_k\gamma_{k-j,0}(1)\E_Q[W_{1,d_V}(P_{Y_{j-1}},Q_{Y^{(j-1)}})],
\end{eqnarray*}
using that $P_{Y^{(j-1)}}=P_{Y_{j-1}}$ by the strong Markov property. Considering $\alpha=L=(L_1,\ldots,L_n)$,   the identity $\tilde W_{L,d_{V}}(P_x,Q)=W_{1,d_{V,L}}(\P_x,Q)$ holds and we have
\[
W_{1,d_{V,L}}(\P_x,Q)\le \sum_{j=1}^n\sum_{k=j}^nL_k\gamma_{k-j,0}(1)\E_Q[W_{1,d_V}(P_{Y_{j-1}},Q_{Y^{(j-1)}})].
\]
For any $1\le j\le n$, let $\pi_j$ denote a conditional coupling scheme $(X'_j,Y_j')$ with marginals $P_{Y_{j-1}}$ and $Q_{Y^{(j-1)}}$ and denote $c_j=\sum_{k=j}^nL_k\gamma_{k-j,0}(1)$.
We estimate the $n$ terms in the first sum of the RHS applying successively the Cauchy-Schwarz and Young inequalities with $\lambda>0$,
\begin{align*}
c_jW_{1,d_V}(P_{Y_{j-1}},Q_{Y^{(j-1)}})=&c_j\inf_{\pi_j}\E_{\pi_j}[(V(X'_j)+V(Y'_j))\1_{X_j'\neq Y_j'}]\\
\le&\inf_{\pi_j}\Big\{ \sqrt{c_j^2\E_{\pi_j}[V^2(X'_j)]\E_{\pi_j}[\pi_j(X'_j\neq Y'_j\mid X'_j)^2]}\\
&\qquad+\sqrt{c_j^2\E_{\pi_j}[V^2(Y'_j)]\E_{\pi_j}[\pi_j(X'_j\neq Y'_j\mid Y'_j)^2]}\Big\}\\
\le& \frac{\lambda c_j^2}2(\E_{\pi_j}[V^2(X'_j)]+\E_{\pi_j}[V^2(Y_j')])\\
&\qquad+\frac{\inf_{\pi_j}\{\E_{\pi_j}[\pi_j(X'_j\neq Y'_j\mid X'_j)^2]+\E_{\pi_j}[\pi_j(X'_j\neq Y'_j\mid Y'_j)^2]\}}{2\lambda}.
\end{align*}
As $X'_j\sim P_{Y_{j-1}}$ we identify $\E_{\pi_j}[V^2(X'_j)]=PV_j^2$.
We then use the following improvement of the Marton inequality \cite{marton:1996a} (see Lemma 8.3 of \cite{boucheron:lugosi:massart:2013} combined with Lemma 2 of \cite{samson:2000})
$$
\inf_{\pi_j}\{\E_{\pi_j}[\pi_j(X'_j\neq Y'_j\mid X'_j)^2]+\E_{\pi_j}[\pi_j(X'_j\neq Y'_j\mid Y'_j)^2]\}\le 2\mathcal K(Q_{Y^{(j-1)}},P_{Y_{j-1}}),
$$
where $\mathcal K(Q,P)$ is the Kullback-Leibler divergence between $P$ and $Q$:
$$
\mathcal K(Q,P)=\E_Q[\log(dQ/dP)].
$$
We obtain, for any $1\le j\le n$:
$$
c_jW_{1,d_V}(P_{Y_{j-1}},Q_{Y^{(j-1)}})\le \frac{\lambda c_j^2}2(PV^2_j+\E_{\pi_j}[V^2(Y'_j)])+\lambda^{-1}\mathcal K(Q_{Y^{(j-1)}},P_{Y_{j-1}}).
$$
Combining those inequalities, as $Y'_j\sim Q_{Y^{(j-1)}}$ so that $\E_Q[\E_{\pi_j}[V^2(Y'_j)]]=\E_Q[V^2_j]$, we obtain
\begin{align*}
W_{1,d_{V,L}}(P,Q)&\le \sum_{j=1}^nc_j\E_Q[W_{1,d_V}(P_{Y_{j-1}},Q_{Y^{(j-1)}})]\\ 
&\le \E_Q\left[\sum_{j=1}^n\left(\frac{\lambda c_j^2}2(PV_j^2+V_j^2)+\lambda^{-1}\mathcal K(Q_{Y^{(j-1)}},P_{Y_{j-1}})\right)\right].
\end{align*}
From the identity
$$
\E_Q\Big[\sum_{j=1}^n\mathcal K(Q_{Y^{(j-1)}},P_{Y_{j-1}})\Big]=\mathcal K(Q,\P_x)
$$
we obtain
$$
W_{1,d_{V,L}}(\P_x,Q)\le \E_Q\left[\sum_{k=1}^n\frac{\lambda c_k^2}2(PV_k^2+V_k^2)\right]+\lambda^{-1} \mathcal K(Q,\P_x).
$$
Then we apply the  Kantorovich duality (see for instance \cite{villani:2009}):
$$
W_{1,d_{V,L}}(P_x,Q)=\sup_g \E_Q[g]-\E_x[g]
$$
where $g$ is 1-Lipschitz with respect to the $d_{V,L}$ metric:
$$
|g(x)-g(y)|\le \sum_{k=1}^nL_j(V(x_k)+V(y_k))\1_{x_k\neq y_k}.
$$
Thus, as any $f$ satisfying \eqref{eq:cond} also satisfies such Lipschitz condition, we obtain
$$
\E_Q\Big[\lambda(f-\E_x[f])-K\sum_{k=1}^n\frac{(\lambda c_k)^2}2(PV_k^2+V_k^2)\Big]\le \mathcal K(Q,\P_x).
$$
Choosing the probability measure $Q$ as
$$
dQ\propto \exp\Big(\lambda(f-\E_x[f])-K\sum_{k=1}^n\frac{(\lambda L_k)^2}2(PV_k^2+V_k^2)\Big)d\P_x
$$
we obtain the desired inequality for the trajectory $X$ starting from $x\in E$.\\

To prove the inequality in the stationary case, it is tempting to integrate the inequality \eqref{mres}. However we do not succeed in replacing $\E_x$ by $\E$ in the exponential term. Step 1 of the proof in \cite{dedecker:gouezel:2015} does not apply not in the unbounded case. Instead, one has to use the same reasoning as above, replacing everywhere $\P_x$ by the stationary distribution  $\P$ of the trajectory $(X_1,\ldots,X_n)$. 
Notice that it is then convenient to add an artificial initial point $X_0 = Y_0 = x$
for a fixed point $x\in E$ to the trajectories  $(X_1,\ldots, X_n)$ and $(Y_1,\ldots,Y_n)$; see \cite{djellout:guillin:wu:2004} and \cite{wintenberger:2015} for more details. Moreover, we check that  the same  $\Gamma_{d_V,d_V}$ weak dependence properties still hold for $\P$ as it is a notion conditional to any possible initial value. So we can apply the same reasoning to prove the result in the stationary case.
\end{proof}

\section{Application to non-asymptotic confidence intervals for MCMC algorithms}\label{sec:appl}
In this section  we are considering an approximation of $\int g(x)d\P(x)=\E[g]$ for some unbounded function $g$ and some distribution $\P$, known up to the normalizing constant. The Markov Chain Monte Carlo (MCMC) algorithms generates the approximation $\frac1n\sum_{k=1}^n g(X_k)$ where $(X_k)_{1\le k\le n}$ is a Markov chain admitting $\P$ as its unique stationary distribution. We refer to \cite{roberts:rosenthal:2004} for a survey on MCMC algorithms.
Usually, one has to consider a burn-in period to deal with the bias $|\E[g]-\E_x[g]|$ due to the arbitrary choice of the initial state $x$ of the Markov chain. However, recent algorithms based on regeneration schemes generate  simulations that are automatically stationary;  see \cite{mykland:tierney:yu:1995} and \cite{brockwell:kadane:2005} for instance. We will only focus on such algorithms to avoid the issue of the burn-in period and corresponding quantitative bounds on the bias  $|\E_x[g]-\E[g]|$.

\subsection{Estimation errors for MCMC algorithms}
An interesting case is when $|g|$ is equal to a drift function $|g|=V$. In the stationary case, with $K$ the constant provided in \eqref{eq:K}, we obtain
$$
\E \Big[\exp\Big(\lambda \sum_{k=1}^n(g(X_k)-\E [g])- \lambda^2 \frac{K^2}2\sum_{k=1}^n(Pg_k^{2 }+g_k^{2})\Big)\Big]\le 1.
$$
Notice that the square integrability of $g$ is satisfied if $g^2 $ is also proportional to a Lyapunov function. Then
the mean ergodic theorem applies and we obtain the a.s. convergence
\begin{equation}\label{slln}
\frac{K^2}{2n}\sum_{k=1}^n(Pg^2_k+g^2_k)\to_{n\to\infty}K^2\E[g^2].
\end{equation}
Moreover, the CLT applies and $(\sum_{k=1}^ng(X_k)-\E_x[g])/\sqrt{n}\to^d  \sigma^2(g) N$ where $N\sim\mathcal N(0,1)$ and 
the asymptotic variance $\sigma^2(g)$ can be expressed as
$$
\sigma^2(g)=\v[g^2] +2\sum_{k=1}^\infty\cov[g(X_0),g(X_k)].
$$
Thus, if one could consider the exponential inequality asymptotically, one would obtain
$$
\E[\exp(\lambda \sigma (g) N)]\le \exp\left(\lambda^2 K^2\E [g^2]\right),\qquad \lambda >0. 
$$
The quantity $K^2\E [g^2]$ appears as a natural over-estimator of $\sigma^2(g)/2$.
Similar upper bounds have been derived under the spectral gap condition in \cite{rosenthal:2003} and under the Ricci curvature condition in \cite{joulin:ollivier:2010}. The spectral gap assumption relies on the control of the correlations for any square integrable functions of the Markov chain. The Ricci curvature condition relies on the contraction properties of any Lipschitz functions of the Markov chain. The advantage of the drift condition approach used here is that the constant $K$ is related only with the Lyapunov function $V$. So the estimate could be much sharper if the Lyapunov function can be well chosen, i.e. close to $g$. However, the bad irreducible properties of the coupling scheme make the constant enormous in most applications (see Section \ref{sec:disc} for numerical values). Better upper bounds for the asymptotic variance than $2K^2\E[g^2]$ have already been obtained in \cite{latuszynski:miasojedow:niemiro:2013} by a direct application of the Nummelin scheme on $(X_1,\ldots,X_n)$ (and not on the coupling scheme)  when $g^2=V$. It is an open question if such sharper over-estimators of the asymptotic variance satisfy an empirical Bernstein inequality similar to \eqref{mres}.
It seems that our large over-estimation is partly due to the fact that the rate of convergence in \eqref{slln} can be very slow (eventually $\E[|g|^{2+\delta}]=\infty$ for all $\delta>0$) but also because  the coupling technique used in the proof is very conservative; see the discussion in Section \ref{sec:disc}.

\subsection{Confidence intervals for the regenerative Metropolis algorithm}\label{sec:ci}
We consider the Random Walk Metropolis algorithm to simulate a Markov chain $(X_k)_{1\le k\le n}$ on $E=\R^d$, $d\ge 1$,  with stationary distribution $\P$  and given a function $h$ proportional to the density of $\P$ with respect to the Lebesgue measure. For some continuous symmetric positive density $q$ one simulates $Z_k$ iid  and $U_{k}$ iid uniform on $[0,1]$ and independent of the $Z_k$. Then one computes recursively the Markov chain $X_k$ from the relation
$$
X_{k}=X_{k-1}+Z_{k}\1_{U_k\le \min(1,h(X_{k-1}+Z_k)/h(X_k))},\qquad k\ge 1,\qquad X_0=x.
$$
Mengersen and Tweedie \cite{mengersen:tweedie:1996} provide a sufficient condition (that is almost necessary) on $h$ for the geometric ergodicity of the Random Walk Metropolis algorithm: the $\alpha$ log-concavity in the tails assumption ($\alpha>0$) asserting
the existence of $x_1>0$ such that
\begin{align}\label{logc}
\frac{h(y)}{h(x)}\le \exp(-\alpha (|y|-|x|)),\qquad |y|>|x|>x_1,
\end{align}
where $|\cdot|$ is some norm on $E$.
Let us recall the  result of Theorem 3.2 of  \cite{mengersen:tweedie:1996}:
\begin{thm}\label{th:mt}
If $d=1$, $h$ satisfies \eqref{logc} and $q(x)\le be^{-\alpha |x|}$ for some $\alpha>0$ 
then the Random Walk Metropolis algorithm is geometrically ergodic with the drift function $V(x)=e^{s|x|}$, $s<\alpha$.
\end{thm}
To overcome the bias issue we simulate under the stationary measure using the Regenerative Metropolis algorithm of Brockwell and Kadane \cite{brockwell:kadane:2005} in a simple version (the algorithm 1 in \cite{brockwell:kadane:2005} with $q$ as the re-entry proposal distribution). The algorithm adds an artificial atom to the Random Walk Metropolis Markov chain that has to be removed to obtain the Markov chain $(X_k)_{1\le k\le n}$. The visits to the atom correspond to the state $A=1$. The chain $X_k$ is only updated outside the atom when $A=0$. So the algorithm can be viewed as a clever series of reject sampling steps and Random Walk  Metropolis  steps with   the same stationary distribution $\P$. The drawback of the approach is that it requires more than $n$ steps to obtain $(X_k)_{1\le k\le n}$ because of the rejection steps. To overcome this efficiency issue, one can use a parallelized version of the algorithm; see \cite{brockwell:2006}. The pseudo code of the algorithm is given in Figure \ref{fig:3}.
\begin{figure}[h!]
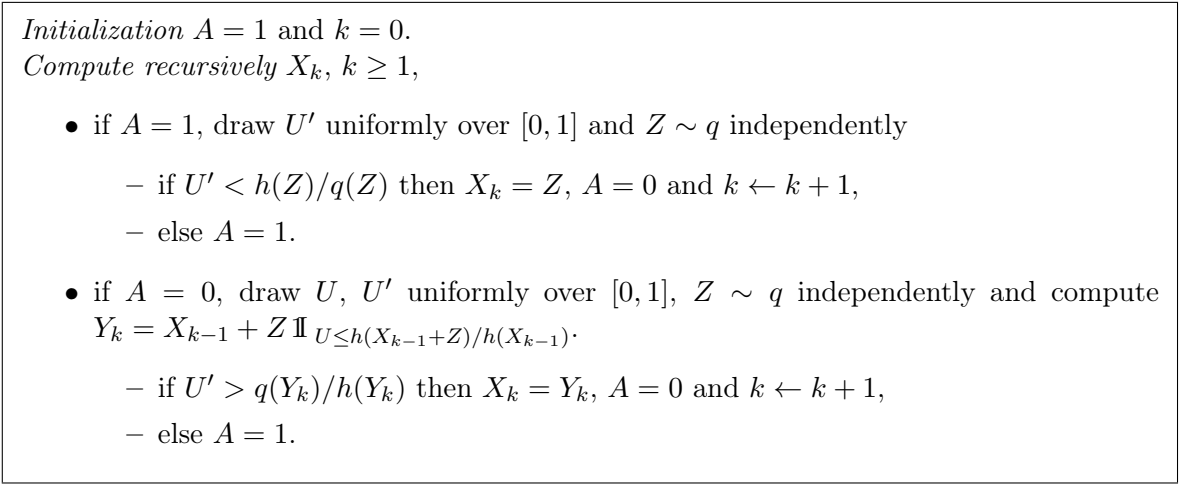

 \fbox{
 \begin{minipage}[c]{15cm}
 \vspace{1mm}
{\it Initialization} $A=1$ and $k=0$.\\
{\it Compute recursively} $X_{k}$, $k\ge 1$,
\begin{itemize}
\item if $A=1$, draw $U'$ uniformly over $[0,1]$ and $Z\sim q$ independently 
    \begin{itemize}
    \item if $U'<h(Z)/q(Z)$ then $X_{k}=Z$, $A=0$ and $k\leftarrow k+1$,
    \item else $A=1$.
\end{itemize}
\item  if $A=0$, draw $U$, $U'$ uniformly over $[0,1]$, $Z\sim q$ independently and compute  $Y_k=X_{k-1}+Z\1_{~U\le h(X_{k-1}+Z)/h(X_{k-1})}$.
   \begin{itemize} 
   \item if $U'>q(Y_k)/h(Y_k)$  then $X_{k}=Y_k$, $A=0$ and $k\leftarrow k+1$,
   \item  else $A=1$.
\end{itemize}
\end{itemize}
\vspace{1mm}
\end{minipage} 
 }
\caption{the regenerative Random Walk Metropolis algorithm}
\label{fig:3}
\end{figure}
Notice that the advantage compared with the reject sampling algorithm is that the rejection steps are more robust to the choice  of the constant  $k>0$ in the threshold $h/(kq)$. Here we fix $k=1$ for simplicity. The algorithm automatically simulates the Markov chain under the stationary measure. It also appears that the rejection step makes the irreducible property of the chain nicer than the one of the chain generated by the Random Walk Metropolis algorithm. Theorem 2 of \cite{brockwell:kadane:2005} applied in our context shows that  the Markov chain satisfies condition \eqref{eq:ir} on $\{V\le R\}$ for any $R>0$ with $\nu$ the marginal distribution after exiting $A=1$ and $c(R)$ the minimum of the probability to obtain $A=1$ from $x\in \{V\le R\}$ and $A=0$:
\begin{align}\nonumber
c(R)&=\min_{\{V\le R\}}\E[q(Y_1)/h(Y_1)\wedge 1\mid X_0=x]\\
\label{eq:c}
&=\min_{\{V\le R\}}\Big\{\E_q\Big[\frac{h(x+Z)}{h(x)}\wedge 1\Big]\E_q\Big[\frac{q(x+Z)}{h(x+Z)}\wedge 1\Big]+\Big(1-\E_q\Big[\frac{h(x+Z)}{h(x)}\wedge 1\Big]\Big)\frac{q(x)}{h(x)}\wedge 1\Big\},
 \end{align}
(larger than the minorization constant given in Lemma 1.2 of \cite{mengersen:tweedie:1996} for the Metropolis algorithm).\\

Define as above the Lyapunov function $V(x)=e^{s|x|}$, $x\in \R$, and 
 denote $\|g\|_V=\sup_{x\in E}|g(x)|/V(x)$. We have the following result, also true for $d\ge 1$,
\begin{thm}
Assume that $h$ satisfies \eqref{logc} and $q(x)\le Ce^{-\alpha |x|}$ for some $C>0$ and $\alpha>2s$. Assume that $R\ge V(x_1)$ is sufficiently large such that 
\[
\bar\beta(R):=\frac{\E_q[\exp(s-\alpha)|Z|)]+1}2+\frac{2V(x_1)\E_q[V(Z)]}{1+R}<1.
\] 
Then for any function $g$ such that $\|g\|_V<\infty$, we have, for any $y>0$ and $n\ge 1$,
\begin{equation}
\left|\frac1n \sum_{k=1}^n g(X_k)-\E [g] \right|\le \frac{x\|g\|_V}{\sqrt n}\sqrt{(\hat \sigma^2_n(V)+y)\left(1+\frac12\log(\hat \sigma^2_n(V)/y+1)\right)},
\end{equation}
with probability $1-\exp(-x^2/2)$, $x>\sqrt 2$ and with the  over-estimator of the asymptotic variance $\sigma^2(V)$:
$$
\hat \sigma^2_n(V):= \left(\frac{1+\bar\beta(R)((R-1)/c(R)-R)}{1-\bar\beta(R)}\right)^2\left(\frac{1+\E_q[V^2(Z)]}n\sum_{k=1}^nV^2(X_k)+\varepsilon_n\right)
$$
and $\varepsilon_n=(\E[V^2(X)]-V_{n}^2\E_q[V^2(Z)])/n$ is considered as a non-observable negligible term.
\end{thm}
\begin{proof}
The minorization condition \eqref{eq:ir} is satisfied for any small set $\{V(x)\le R\}$ with the constant $c(R)$ in \eqref{eq:c}. Let us check that the Markov chain satisfies the drift condition \eqref{eq:geom} with the Lyapunov function $V(x)=\exp(s|x|)$. To do so, notice that by definition the chain is updated when $k$ increases in two cases corresponding to the first and third items in Figure \ref{fig:3}, referred as cases $A=1$ and $A=0$ respectively. First consider the case $A=1$, then $\E_x[V(X_1)]\le \E_q[V(Z)]$, $x\in E$ and $\E_q[V(Z)]=\E_q[\exp(s|Z|)]$ is finite because $q(x)\le be^{-\alpha |x|}$. Second, consider the case $A=0$ and $|x|> x_1$, then under \eqref{logc} we have
\begin{align*}
\E_x[V(X_1)]&=\E_x[V(X_1)\1_{|X_1|\le |x|}]+ \E_x[V(X_1)\1_{|X_1|> |x|}]\\
&\le V(x)P_x(|X_1|\le |x|)+ \E_x\Big[V(x+Z_1) h(x+Z_1)/h(x) \1_{|x+Z_1|> |x|}\Big]\\
&\le \exp(s|x|)\Big(1+\E_q\Big[(\exp((s-\alpha)(|x+Z_1|-|x|))-1)\1_{|x+Z_1|> |x|}\Big]\Big).
\end{align*} 
If $x>0$, as the integrand is negative we have:
\begin{align*}
\E_q\Big[(\exp((s-\alpha)(|x+Z_1|-|x|))-1)\1_{|x+Z_1|> |x|}\Big]\le \E_q \Big[(\exp((s-\alpha)Z_1)-1)\1_{ Z_1>0}\Big].
\end{align*}
The same reasoning applies if $x<0$ and as $q$ is symmetric we obtain
\begin{align*}
\E_q \Big[(\exp((s-\alpha)(|x+Z_1|-|x|))-1)\1_{|x+Z_1|> |x|}\Big]\le \frac{\E_q [\exp((s-\alpha)|Z_1|)-1 ]}2.
\end{align*}
Finally, when $A=0$ and $|x|\le x_1$ we use the upper bound $\E_x[V(X_1)]\le V(x_1)\E[V(Z)]$.
Thus, the drift condition \eqref{eq:geom} is satisfied by $V(x)=e^{s|x|}$ with $b_1=V(x_1) \E_q[V(Z)]$ and 
$\beta_1=(\E_q[\exp(s-\alpha)|Z|)]+1)/2$. Notice that by similar arguments we also have the drift condition \eqref{eq:geom} satisfied by $V^2$ with $b_2= V^2(x_1)\E_q[V^2(Z)]$ and 
$\beta_2=(1+\E_q[\exp(2s-\alpha)|Z|)])/2$. So   second order moments are finite and the quantities $PV_k^2$ are well defined.  We apply the stationary version of Theorem \ref{th:main} to obtain
\begin{align*}
\E \Big[\exp\Big(\lambda \sum_{k=1}^n(g(X_k)-\E [g])-\frac{\lambda^2}2K^2\sum_{k=1}^n(PV_k^{2 }+V_k^{2})\Big)\Big]\le 1.
\end{align*}
As $PV_k^2$ is not observed, we over-estimate it by $V_{k-1}^2\E[V^2(Z)]$ for $2\le k\le n$. The negligible term $\varepsilon_n$ correspond to the fact that $PV_1^2=\E[V^2(X)]$ is replaced by $V_{n}^2\E_q[V^2(Z)]$ in the expression of $\hat\sigma^2_n(V)$. Finally, we apply Corollary 2.2 of \cite{delapena:klass:lai:2004} to obtain the desired result.
\end{proof}

\section{Discussion and simulations study}\label{sec:disc}

In this section we discuss on the application of Section \ref{sec:ci} along with a simulation study. We would like to stress the fact that the new deviation inequality \eqref{mres} has many other applications in mathematical statistics that will be addressed in future work.\\

{\it Discussion about the Lyapunov function $V$:} Compared with \cite{dedecker:gouezel:2015}, the approach is very dependent on the choice of the Lyapunov function $V$. The constants involved in the drift condition \eqref{eq:geom} can be reasonable if $V$ is well chosen. Moreover, for the MCMC application when $f(X_1,\ldots,X_n)=\sum_{k=1}^n g(X_k)$, it seems more efficient to take $V$ as close to $|g|$ as possible, i.e. as small as possible. Indeed, the larger $V$, the larger $b$ in \eqref{eq:geom}. By a convexity argument, one can actually show that the drift condition \eqref{eq:geom} holds for all Lyapunov's functions $V^p$ with $0<p<1$. So the range of admissible Lyapunov functions is quite large. For instance, in the Metropolis algorithm, any $V(x)=\exp(s|x|)$ for $s<2\alpha$ is admissible. However, we are not aware of any other Lyapunov functions for this algorithm and the Metropolis algorithm seems to have good properties for functions $g$ with exponential shape only. An interesting issue is to know wether, given an unbounded $g$, one can always find an algorithm such that \eqref{eq:geom} is satisfied for some Lyapunov function close to $|g|$.\\

{\it Discussion about the quantitative bounds:} The explicit constant $K$ in Theorem \ref{th:main} is very large. For instance, the contracting normal toy-example considered in \cite{baxendale:2005} satisfies our conditions; it corresponds to the case of an AR(1) model $X_{k}=0.5X_{k-1}+\sqrt{3/4}N_k$ where the $N_k$s are iid standard Gaussian random variables. The stationary solution is the standard gaussian distribution, $g(x)=x$, $\E[g]=0$ and $V(x)=1+x^2$; see \cite{baxendale:2005} and \cite{latuszynski:miasojedow:niemiro:2013} for more details. Then the constant $K=(1+2\bar\beta(R)((R-1)/c(R)-R)/(1- \bar \beta(R))\approx 7,000,000,000$, is larger by 3 orders of magnitude than the constants in \cite{latuszynski:niemiro:2011}. Note that \cite{latuszynski:miasojedow:niemiro:2013} improved our constants by 5 orders of magnitude, i.e. half less. Our bounds are much larger because of the use of the coupling argument, moving from a univariate problem to a bivariate one. It would be interesting to obtain an empirical Bernstein inequality by applying the marginal Nummelin scheme directly on $(X_1,\ldots,X_n)$.\\

More precisely, the enormous constant $K$ is due to the poor irreducibility properties of the toy-example,$$
c(R)=2(\Phi(\sqrt 3 d)-\Phi(\sqrt 3/ d))\qquad \mbox{with}\qquad R=\sqrt{2+(d^2-1)/4};
$$
 see \cite{baxendale:2005} and \cite{latuszynski:niemiro:2011} for details on this elementary computation. As small values of $c(R)$ are the main issue to control the constant, it is worth  to improve the irreducibility properties of the Markov  chain. Regenerative algorithms as the one of Brockwell and Kadane \cite{brockwell:kadane:2005} offer a simple way of increasing $c(R)$. The only drawback is that it requires more steps to generate a trajectory of fixed length. In Figure \ref{fig:2} we compare the distributions of the outputs of the Regenerative, Rejection and Metropolis algorithms based on $10000$ Monte Carlo simulations of $10000$ runs. 
\begin{figure}[h!]
\centering
\includegraphics[width=14cm,height=7cm]{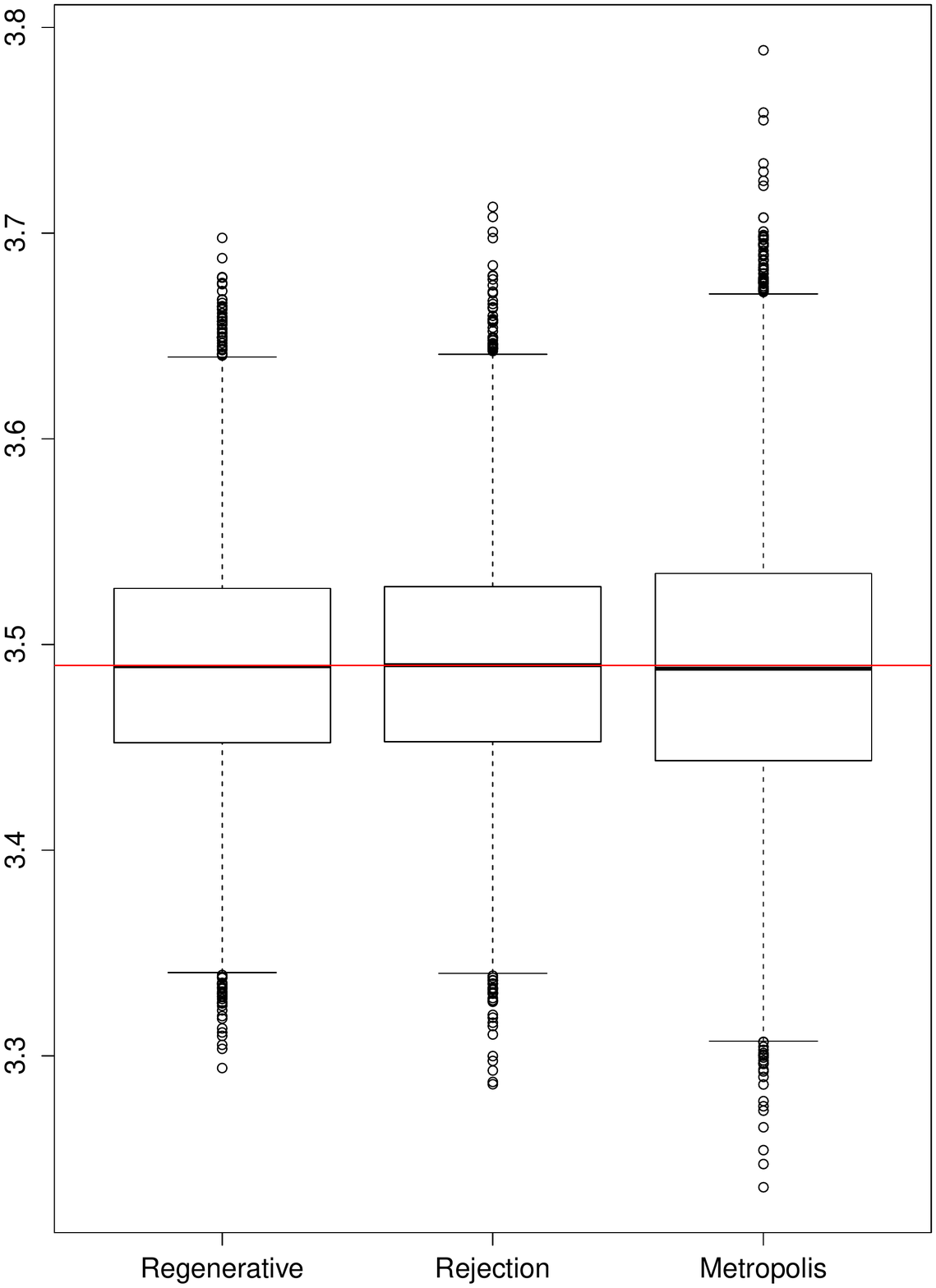}
\caption{Boxplots of the outputs of the Regenerative, Rejection and Metropolis algorithms based on 10000 runs with averaged $n=8308$, $=2532$ (optimal, not tractable in practice) and $=10000$ respectively. In red line is the true value.}
\label{fig:2}
\end{figure} 
The proposal distribution is the standard Gaussian ($d=1$) and $h(x)=e^{-(x-1)^2}$, $x\in\R$. The initial value for the Metropolis algorithm is $0$. The bias issue could  explain why the Metropolis algorithm is slightly over-performed by the Regenerative algorithm. The large number of rejects should explain why the Rejection algorithm is slightly over-performed by the Regenerative algorithm, even if the reject ratio has been optimized in the Rejection algorithm (which is not possible in practice) and not in the Regenerative algorithm $(k=1)$. From \eqref{eq:c}, we have $c(R)\ge \inf_x q(x)/h(x)=(e\sqrt {2\pi})^{-1}\approx 0.15$ that is reasonable. Optimizing in $K$ on $R$ we obtain $K\approx 40,000$. It still requires more than $100*\log(10)*K^2*\log(K)/2\approx 3,800,000,000,000$ runs for obtaining a confident interval of level $0.1$ and of reasonable length $\approx \sigma (V)/10$.\\

{\it Discussion about the median trick:} We based our comparison with previous quantitative bounds of \cite{latuszynski:niemiro:2011} and \cite{latuszynski:miasojedow:niemiro:2013} on confident intervals of level $\sigma (V)/10$. As the previous bounds \cite{latuszynski:niemiro:2011} and \cite{latuszynski:miasojedow:niemiro:2013} are based on the Chebychev inequality
$$
\P\left(\left|\sum_{k=1}^n g(X_k)-\E[g]\right|>\varepsilon\right)\le \frac{\|g\|_V^2\hat\sigma_n^2(V)}{n\varepsilon},
$$
they are not efficient to produce confidence intervals with small levels. To bypass the problem, the median trick of \cite{jerrum:valiant:vizirani:1986} is used. The trick is to approximate $\E[g]$ thanks to the median of $m$ independent approximations $\frac1n\sum_{k=1}^n g(X_{i,k})$, $1\le i \le m$ of MCMC algorithms with the same confidence interval length of level $a<1/2$. Then if $m\ge 2\log(\alpha)/\log(4a(1-a))$ the confidence level of the interval around the median is reduced to $\alpha<a$, see Lemma 4.4 in \cite{latuszynski:niemiro:2011}. However, empirical Bernstein's inequalities as \eqref{mres} show that the interval around the mean of the $m$ independent approximations (based on $mn$ runs) has level $\alpha<a$ when $m\ge \log(\alpha)/\log(a)$.
So, when Theorem \ref{th:main} applies, the mean $\frac1m\sum_{i=1}^m\frac1n\sum_{k=1}^n g(X_{i,k})$ seems to have better concentration properties than the median.
\section*{Acknowledgments}
{I am grateful to an anonymous referee for helpful comments. I would also like to thank Thomas Mikosch for polishing the writing of the final version. Finally, I would like to dedicate this paper to my great supervisor Paul Doukhan.}

\end{document}